\documentclass[12pt,a4paper]{amsart}
\makeatletter
\renewcommand\normalsize{%
    \@setfontsize\normalsize{11.7}{14pt plus .3pt minus .3pt}%
    \abovedisplayskip 10\p@ \@plus4\p@ \@minus4\p@
    \abovedisplayshortskip 6\p@ \@plus2\p@
    \belowdisplayshortskip 6\p@ \@plus2\p@
    \belowdisplayskip \abovedisplayskip}
\renewcommand\small{%
    \@setfontsize\small{9.5}{12\p@ plus .2\p@ minus .2\p@}%
    \abovedisplayskip 8.5\p@ \@plus4\p@ \@minus1\p@
    \belowdisplayskip \abovedisplayskip
    \abovedisplayshortskip \abovedisplayskip
    \belowdisplayshortskip \abovedisplayskip}
\renewcommand\footnotesize{%
    \@setfontsize\footnotesize{8.5}{9.25\p@ plus .1pt minus .1pt}%%
    \abovedisplayskip 6\p@ \@plus4\p@ \@minus1\p@
    \belowdisplayskip \abovedisplayskip
    \abovedisplayshortskip \abovedisplayskip
    \belowdisplayshortskip \abovedisplayskip}
\ifdefined\pdfpagewidth
\else
\fi
\calclayout
\makeatother

\usepackage{eucal}

\usepackage{graphicx} 

\usepackage{comment}
\usepackage{color}
\usepackage{amssymb,graphicx, stmaryrd}
\usepackage[normalem]{ulem}

\usepackage{amsmath,amssymb}  
\usepackage{xparse}

\DeclareRobustCommand{\g}{\ensuremath{\mathrm{g}}}

\usepackage{ytableau} 
\ytableausetup{boxsize=1em}

\usepackage[colorlinks,  linkcolor=blue, citecolor=blue, linktocpage,urlcolor=blue]{hyperref}

\theoremstyle{definition}
\newtheorem{theorem}{Theorem}[section]

\newtheorem{lemma}[theorem]{Lemma}

\newtheorem{remark}[theorem]{Remark}

\newtheorem{question}[theorem]{Question}

\author{Kyle Broder}

\title{Cohomogeneity-one Einstein metrics on \\ closed $4$-manifolds}

\usepackage{fancyhdr}
\usepackage{amssymb}

\usepackage{graphicx}

\usepackage{tikz}
\usepackage{tikz-cd}
\usepackage{float}

\usepackage{amsmath,amssymb,amsthm,mathtools}
\thanks{\em{k.broder@uq.edu.au}.}
\thanks{The University of Queensland,  St. Lucia,  QLD 4067, Australia}


\providecommand{\noopsort}[1]{}
\begin{document}

\maketitle

\begin{abstract}
    We classify the Einstein metrics on closed 4-manifolds whose isometry group has dimension at least 3. Consequently, we show that the cohomogeneity-one Einstein metrics on a closed $4$-manifold are locally symmetric or homothetic to the Page metric on $\mathbf{CP}^2 \sharp \overline{\mathbf{CP}}^2$.
\end{abstract}

\section{\textbf{Introduction}}
\noindent A Riemannian metric g on a closed manifold $M^n$ is \emph{Einstein} if \begin{eqnarray*}\label{eqn:Einstein}
    \text{Ric}(\g) &=& \lambda \g, \qquad \lambda \in \mathbf{R}.
\end{eqnarray*} One of the most effective mechanisms for producing Einstein metrics with $\lambda>0$ comes from considering \emph{cohomogeneity-one} actions---that is, a compact Lie group $\mathsf{G}$ acts isometrically on $(M,\g)$ with one-dimensional orbit space (see, e.g., \cite{Ziller2009GeometryCohomogeneityOne}). B\"ohm \cite{Bohm1998InhomogeneousEinstein} used cohomogeneity-one techniques to produce the first inhomogeneous Einstein metrics on spheres. Moreover, to date, the only non-round Einstein metrics on even-dimensional spheres have been found using cohomogeneity-one methods \cite{FoscoloHaskins2017, Chi2024, NienhausWink2025TenSphere, ButtsworthHodgkinson2026S12}.

On closed $4$-manifolds, methods from complex and symplectic geometry have led to several classification results for Einstein metrics. LeBrun \cite{LeBrun2012EinsteinHermitian} proved that a Hermitian Einstein metric on a closed \(4\)-manifold is either K\"ahler--Einstein or, up to rescaling and isometry, one of two exceptional conformally K\"ahler examples: the Page metric \cite{Page1978CompactRotatingInstanton} on \(\mathbf{CP}^2\sharp\overline{\mathbf{CP}}^{2}\) and the
Chen--LeBrun--Weber metric \cite{ChenLeBrunWeber2008ConformallyKahlerEinstein} on \(\mathbf{CP}^2\sharp 2\overline{\mathbf{CP}}^{2}\). Earlier, LeBrun \cite{LeBrun1997EinsteinMetricsComplexSurfaces} showed that a compact complex surface admitting a non-K\"ahler Einstein Hermitian metric must be the blow-up of \(\mathbf{CP}^2\) at one, two, or three points, and that the isometry group of the metric contains a two-torus. The closed $4$-manifolds admitting K\"ahler--Einstein metrics with $\lambda>0$ were classified by Tian \cite{Tian1987KahlerEinsteinCertain, Tian1990CalabiConjectureSurfaces}.

In contrast, general classification results for cohomogeneity-one Einstein metrics are scarce. The problem of classifying cohomogeneity-one Einstein metrics is raised explicitly by Anderson \cite[page 6]{Anderson2010SurveyEinsteinMetrics} (cf. \cite{dammerman2004metrics,Donovan2025,Seshadri2004EinsteinFourManifolds,Wang2026ComputerAssistedEinsteinMetrics}). One of the strongest previous rigidity results is due to B\'erard-Bergery and Derdzi\'nski \cite{BerardBergery1982NouvellesEinstein,Derdzinski1983SelfDualKahler} (see also \cite[Theorem 9.127]{besse1987einstein} and \cite{PetersenZhu1995}). It implies that if \((M^4, \g)\) is closed, simply connected, and Einstein, and either \(\dim \mathsf{Isom}(M,\g)\geq 4\) or \(M\) admits an isometric \(\mathsf{SO}(3)\)-action with principal orbits \(S^2\) or \(\mathbf{RP}^2\), then \((M,\g)\) is either locally symmetric or homothetic to the Page metric. Recent work has continued to develop numerical and computer-assisted approaches to this problem \cite{Donovan2025,Wang2026ComputerAssistedEinsteinMetrics}.

The main result of this paper is the following classification theorem.

\begin{theorem}\label{thm:Main-theorem}
    Let $(M^4,\g)$ be a connected closed Einstein $4$-manifold with $\dim \mathsf{Isom}(M,\g) \geq 3$. Then either \(\g\) is flat or the universal cover is homothetic to one of the following: \begin{eqnarray*}
        (S^4,\g_{\text{rd}}), \qquad (\mathbf{CP}^2,\g_{\text{FS}}), \qquad (S^2 \times S^2, \g_{\text{rd}} \oplus \g_{\text{rd}}), \qquad (\mathbf{CP}^2 \sharp \overline{\mathbf{CP}}^2, \g_{\text{Page}}).
    \end{eqnarray*} 
\end{theorem}

\vspace{0.2cm}

\noindent In particular, $S^4$ is the first even-dimensional sphere of dimension $n \geq 4$ for which the round metric is the only cohomogeneity-one Einstein metric, up to scaling (cf. \cite{Bohm1998InhomogeneousEinstein, FoscoloHaskins2017, Chi2024, NienhausWink2025TenSphere, ButtsworthHodgkinson2026S12}). 

In the more general category of orbifold metrics, Hitchin \cite{Hitchin1996NewFamilyEinsteinMetrics} discovered, for each integer $k>0$, an $\mathsf{SO}(3)$-invariant cohomogeneity-one self-dual Einstein orbifold metric $\g_{\text{H},k}$ on an orbifold $\mathcal{O}_k$ homeomorphic to $S^4$ (cf. \cite{GroveWilkingZiller2008, Ziller2009GeometryCohomogeneityOne}).  In the final section of the paper, we prove the following result. 

\begin{theorem}\label{thm:orbifold}
For an integer $k \geq 1$, let $\g$ be an $\mathsf{SO}(3)$-invariant Einstein orbifold metric on the Hitchin orbifold \( \mathcal{O}_k\). Then g is homothetic to $\g_{\text{H},k}$.
\end{theorem}

\vspace{0.2cm}

\noindent In Section~\ref{sec:2}, we show that the proof of Theorem~\ref{thm:Main-theorem} reduces to the classification of cohomogeneity-one Einstein metrics with $\mathsf{G} = \mathsf{SU}(2)$ or $\mathsf{G} = \mathsf{SO}(3)$. On the union of the principal orbits, a $\mathsf{G}$-invariant Einstein metric is then described by three functions subject to boundary conditions. The key insight is to consider the induced system of evolution equations for the connection coefficients on the bundles of self-dual and anti-self-dual $2$-forms, and for the eigenvalues of the curvature operator on $\Lambda_{\pm}^2$. These evolution equations allow us to treat the $\mathsf{SO}(3)$-actions (see Sections~\ref{subsec:SO(3)-S4}, \ref{subsec:SO(3)-CP2}, and \ref{sec:next}). The \(\mathsf{SU}(2)\)-actions on \(S^4\) and \(\mathbf{CP}^2\) are handled by a maximum principle argument (see Sections~\ref{subsec:SU(2)-S4} and \ref{subsec:SU(2)-CP2}). The remaining $\mathsf{SU}(2)$-action on $\mathbf{CP}^2 \sharp \overline{\mathbf{CP}}^2$ is also handled by a maximum principle argument after changing the group diagram presentation (see Section~\ref{sec:Trapping}).

We close the introduction by mentioning that significant progress in the classification of Einstein metrics has come from imposing additional curvature constraints (see, e.g., \cite{GurskyLeBrun1999, Tachibana1974PositiveCurvatureOperator, Brendle2010, CaoTran2018, CaoTran2022FourManifolds, VerdianiZiller2021, PetersenWink2022, Liu2025ToricEinstein4Manifolds}).

\subsection*{Acknowledgments}
I am indebted to Timothy Buttsworth. After I proved Theorem~\ref{thm:SU(2)-CP2}, he encouraged me to look for further results, which ultimately led to the classification theorem proved in this paper. I owe to Patrick Donovan the idea of using the K\"ahler condition to handle the \(\mathsf{SO}(3)\)-actions on \(\mathbf{CP}^2\) and \(S^2 \times S^2\). He also suggested the formulation of Theorem~\ref{thm:Main-theorem}. The proof of Theorem~\ref{thm:SU(2)-CP2-blow-up} was drastically simplified following suggestions from Wolfgang Ziller. Ziller also encouraged me to prove Theorem~\ref{thm:orbifold}. I also thank Christoph B\"ohm, Timothy Buttsworth, Peter Petersen, Matthias Wink, and Wolfgang Ziller for very detailed comments on earlier versions of the manuscript. 
 
\section{\textbf{Preliminaries}}\label{sec:2}

\noindent We first reduce Theorem~\ref{thm:Main-theorem} to the classification of cohomogeneity-one Einstein manifolds. The assumption on the isometry group implies that $\lambda \geq 0$, and if $\lambda=0$, the metric is flat (see, e.g., \cite[Section 8.2]{petersen2016riemannian}). So assume $\lambda>0$. Since $\pi_1(M)$ is finite, after passing to the universal cover, we may assume that $M$ is simply connected. Write $\mathsf{G} := \mathsf{Isom}_0(M,\g)$ for the identity component of the isometry group. From \cite{BerardBergery1982NouvellesEinstein, Derdzinski1983SelfDualKahler}, we may assume $\dim \mathsf{G} =3$. Let $\mathsf{G}/\mathsf{H}$ be a principal orbit. Since $\mathsf{H}$ is principal, the slice representation is trivial. Thus, the induced $\mathsf{G}$-action on $\mathsf{G}/\mathsf{H}$ is effective. In particular, if $d : = \dim \mathsf{G}/\mathsf{H} \leq 2$, then $\mathsf{G}$ acts effectively and transitively on a compact $d$-dimensional manifold. So $d=2$ and $\mathsf{G}/\mathsf{H}$ is a compact homogeneous surface with a three-dimensional connected isometry group. Hence, $\mathsf{G}=\mathsf{SO}(3)$ and $\mathsf{G}/\mathsf{H}$ is $S^2$ or $\mathbf{RP}^2$, which is again covered by \cite{BerardBergery1982NouvellesEinstein, Derdzinski1983SelfDualKahler}.  It remains to consider the effective cohomogeneity-one actions with $\dim \mathsf{G}=3$. By the Parker--Hoelscher classification \cite{Parker1986, Hoelscher2010} (see also \cite{Donovan2025}), Theorem~\ref{thm:Main-theorem} is reduced to understanding the Einstein metrics invariant under the $\mathsf{SU}(2)$ actions on $S^4$, $\mathbf{CP}^2$, $\mathbf{CP}^2 \sharp \overline{\mathbf{CP}}^2$, and the standard $\mathsf{SO}(3)$-actions on $S^4$, $\mathbf{CP}^2$, and $S^2 \times S^2$. This will occupy Sections~\ref{sec:2}--\ref{sec:next}.

\begin{remark}
    Let us remark that if $(M,\g)$ is a closed simply connected Einstein $4$-manifold with $\dim \mathsf{Isom}(M,\g)>0$, then the same argument as before shows that $\lambda>0$ or $\lambda=0$, in which case, the metric is flat. If $\dim \mathsf{Isom}(M,\g) =2$, the metric has cohomogeneity two. The closed simply connected $4$-manifolds admitting an effective cohomogeneity-two action are classified by Orlik--Raymond \cite{OrlikRaymond1970TorusActionsI}. They are equivariantly diffeomorphic to connected sums of $\mathbf{CP}^2$, $\overline{\mathbf{CP}}^2$, $S^2 \times S^2$. From \cite{Liu2025ToricEinstein4Manifolds}, those known to admit invariant Einstein metrics are those listed in Theorem~\ref{thm:Main-theorem} in addition to the Chen--LeBrun--Weber metric \cite{ChenLeBrunWeber2008ConformallyKahlerEinstein} on $\mathbf{CP}^2 \sharp 2\overline{\mathbf{CP}}^2$ and the unique K\"ahler--Einstein metric  \cite{Tian1987KahlerEinsteinCertain, Tian1990CalabiConjectureSurfaces} on $\mathbf{CP}^2 \sharp 3 \overline{\mathbf{CP}}^2$.

    If $\dim \mathsf{Isom}(M,\g)=1$, then the manifold admits a cohomogeneity-three circle action. From \cite{Fintushel1978ClassificationCircleActions}, a locally smooth $\mathsf{S}^1$-action on a simply connected closed $4$-manifold is diffeomorphic to a connected sum of $S^4$, $\mathbf{CP}^2$, $\overline{\mathbf{CP}}^2$, or $S^2 \times S^2$. However, very little seems to be known about the Einstein metrics in this case (see, e.g., \cite{Anderson2010SurveyEinsteinMetrics, Seshadri2004EinsteinFourManifolds}).
\end{remark}

\subsection{Structure of cohomogeneity-one manifolds}
Following \cite{GroveZiller2011Lifting, BettiolKrishnan2019}, let $(M,\g)$ be a closed simply connected $4$-manifold. A group $\mathsf{G}$ acting isometrically on $(M,\g)$ is said to have \emph{cohomogeneity-one} if the orbit space $M/\mathsf{G}$ is diffeomorphic to $[0,T]$ (see, e.g., \cite{GroveZiller2011Lifting, Ziller2009GeometryCohomogeneityOne, BettiolKrishnan2019, Donovan2025}). If $\pi : M \to [0,T]$ denotes the quotient map, the \emph{principal orbits} are the codimension one orbits $\pi^{-1}(t)$ for $0<t<T$, and the \emph{singular orbits} are $\pi^{-1}(0)$ and $\pi^{-1}(T)$. 

With respect to the quotient metric, the quotient $M / \mathsf{G}$ is isometric to the interval $[0,T]$. Let $c : [0,T] \to M$ be a geodesic orthogonal to the orbits. For $t_0, t \in (0,T)$, the \emph{principal isotropy} $\mathsf{H} : = \mathsf{G}_{c(t_0)} \simeq \mathsf{G}_{c(t)}$. The two singular orbits have isotropy $\mathsf{K}_- := \mathsf{G}_{c(0)}$ and $\mathsf{K}_+ := \mathsf{G}_{c(T)}$, and $M$ is equivariantly diffeomorphic to \begin{eqnarray}\label{eqn:disk-bundle}
    \mathsf{G} \times_{\mathsf{K}_-} D^{\ell_- +1} \cup_{\mathsf{G}/\mathsf{H}} \mathsf{G} \times_{\mathsf{K}_+} D^{\ell_+ +1},
\end{eqnarray} where $S^{\ell_{\pm}} = \partial D^{\ell_{\pm} +1} \simeq \mathsf{K}_{\pm} / \mathsf{H}$. The groups $\mathsf{G}, \mathsf{H}, \mathsf{K}_{\pm}$ together with their embeddings $j_{\pm} : \mathsf{K}_{\pm} \to \mathsf{G}$ and $i_{\pm} : \mathsf{H} \to \mathsf{K}_{\pm}$ define a \emph{group diagram} that is denoted by $\mathsf{H} \subset \{ \mathsf{K}_-, \mathsf{K}_+\} \subset \mathsf{G}$. From \eqref{eqn:disk-bundle}, a group diagram determines $M$. Conversely, a diagram $\mathsf{H} \subset \{ \mathsf{K}_-, \mathsf{K}_+ \} \subset \mathsf{G}$, together with the embeddings and with $\mathsf{K}_{\pm} / \mathsf{H}$ homogeneous spheres whose actions extend over the disks $D^{\ell_{\pm}+1}$, defines a cohomogeneity-one $\mathsf{G}$-manifold by the above disk bundle construction. 

However, the equivariant diffeomorphism type is represented not by a unique diagram, but by the equivalence class generated by simultaneous conjugation in $\mathsf{G}$, by replacing $\mathsf{K}_+$ with $n \mathsf{K}_+ n^{-1}$ for $n \in N_{\mathsf{G}}(\mathsf{H})_0$, and, if the endpoints are unlabeled, by interchanging $\mathsf{K}_-$ and $\mathsf{K}_+$. That is, for $a \in \mathsf{G}$ and $n \in N_{\mathsf{G}}(\mathsf{H})_0$, the manifolds defined by the group diagrams $\mathsf{H} \subset \{ \mathsf{K}_-, \mathsf{K}_+ \} \subset \mathsf{G}$ and $a\mathsf{H}a^{-1} \subset \{ a \mathsf{K}_-a^{-1}, an\mathsf{K}_+ n^{-1} a^{-1} \} \subset \mathsf{G}$ are equivariantly diffeomorphic (see \cite{GroveWilkingZiller2008, GroveZiller2011Lifting}).

The closed simply connected $4$-manifolds admitting a cohomogeneity-one action are $S^4$, $\mathbf{CP}^2$, $S^2 \times S^2$, and $\mathbf{CP}^2 \sharp \overline{\mathbf{CP}}^2$. This classification was first obtained by Parker \cite{Parker1986} with the omission of the $\mathsf{SO}(3)$-action on $\mathbf{CP}^2$ and all of the $\mathsf{U}(2)$-actions. The omissions were corrected by Hoelscher \cite{Hoelscher2010}, who classified the compact connected almost effective non-reducible actions. The complete classification of the group diagrams is due to Grove--Ziller \cite{GroveZiller2011Lifting}.

\subsection{Cohomogeneity-one Einstein metrics}
Let $(M,\g)$ be a simply connected closed $4$-manifold. Let $\mathsf{G}$ be a fixed compact connected three-dimensional group, $\mathsf{G} = \mathsf{SU}(2)$ or $\mathsf{SO}(3)$, acting almost effectively and isometrically on $M$ with cohomogeneity-one. Let $(\sigma_1,\sigma_2,\sigma_3)$ be a fixed left-invariant coframe on $\mathsf{G}$, with \begin{eqnarray}\label{eqn:structure-equations-sigma}
    \text{d}\sigma_i &=& - 2 \sigma_j \wedge \sigma_k,
\end{eqnarray} where $(i,j,k)$ is cyclic.  Since $\mathsf{G}$ acts on $(M,\g)$ by isometries, the $\mathsf{G}$-invariant metric is determined by the one-parameter family of $\mathsf{H}$-invariant inner products $\g_t$ on $T_{c(t)}(\mathsf{G} \cdot c(t))$, together with the normal arc length coordinate $t$, where $c(t)$ is the geodesic meeting all orbits orthogonally. Further, it suffices to determine g on the open dense subset given by the union of the principal orbits $M^{\circ} := \bigcup_{t \in (0,T)} \pi^{-1}(t) \subset M$. On $M^{\circ}$, we can write $\g = \text{d} t^2 + \g_t$, where $\g_t$ is a one-parameter family of homogeneous metrics on $\mathsf{G}/\mathsf{H}$. Conversely, a metric of the form $\g = \text{d}t^2+\g_t$ defines a cohomogeneity-one metric on $M$; certain smoothness conditions must be satisfied at the endpoints $t=0$ and $t=T$.

Let $\{ X_i \}$ be a basis of the Lie algebra of $\mathsf{G}$, adapted to the inclusion $\mathsf{H} \subset \{ \mathsf{K}_-, \mathsf{K}_+\} \subset \mathsf{G}$ and orthonormal with respect to a fixed bi-invariant metric. Let $X_i^{\ast}(t) : = \frac{\text{d}}{\text{d}s} \exp(s X_i) \cdot c(t) \vert_{s=0}$. A metric is \emph{diagonal} if on $M^{\circ}$, it can be written in the form \begin{eqnarray}\label{eqn:metric-form}
    \g &=& \text{d}t^2 + f_1^2(t) \sigma_1^2 + f_2^2(t) \sigma_2^2 + f_3^2(t) \sigma_3^2, \qquad 0 < t < T.
\end{eqnarray} In this case, the $f_i(t)$ are the lengths of the Killing fields $X_i^{\ast}(t)$, and these Killing fields vanish at $t=0$ or $t=T$ if and only if $X_i$ belongs to the Lie algebra of the corresponding isotropy subgroup $\mathsf{K}_{\pm}$. In this case, the smoothness conditions translate into conditions on the Taylor series of $f_i(t)$ at $t=0$ and $t=T$ (see \cite{EschenburgWang2000InitialValue, GroveZiller2002, GroveVerdianiZiller2011, VerdianiZiller2022}).

\begin{lemma}\label{lem:diagonal}
Let g be a $\mathsf{G}$-invariant cohomogeneity-one Einstein metric. If $\mathsf{G} = \mathsf{SU}(2)$ or $\mathsf{G}=\mathsf{SO}(3)$, then g can be written in diagonal form. 
\end{lemma} \begin{proof}
    For the $\mathsf{SO}(3)$-action on $S^4$, the principal isotropy is $\mathsf{H} \simeq \mathbf{Z}_2 \oplus \mathbf{Z}_2$, and $\mathfrak{so}(3)$ splits into three distinct one-dimensional characters of $\mathsf{H}$. In particular, every $\mathsf{SO}(3)$-invariant metric is diagonal. For the $\mathsf{SU}(2)$-actions, the invariant Einstein metrics are diagonal \cite[Theorem 3.1]{Donovan2025} (see also \cite{Buttsworth2025SU2SteadyRicciSolitons}). For the $\mathsf{SO}(3)$-actions on $\mathbf{CP}^2$ and $S^2 \times S^2$, the Einstein metrics are diagonal \cite{Dammerman2009}. 
\end{proof}

\begin{remark}\label{remark:2.4}
Let g be an invariant metric of the form \eqref{eqn:metric-form}. If $f_i \equiv f_j$, then the pulled-back metric on $(0,T) \times \mathsf{G}$ is invariant under the right action of $\eta_k : = \exp(\mathbf{R} X_k)$, since $\text{Ad}(\eta_k)$ rotates $\text{span} \{ X_i, X_j \}$ and fixes $X_k$. The right action descends to $(0, T) \times \mathsf{G}/\mathsf{H}$ if and only if $\eta_k \subset \mathsf{N}_{\mathsf{G}}(\mathsf{H})$, and it extends across the singular orbits if and only if, for the chosen representative of the group diagram, it preserves $\mathsf{K}_{\pm}$. For the $\mathsf{SU}(2)$-actions considered below, this implies that an $\mathsf{SU}(2)$-invariant Einstein metric with $f_i \equiv f_j$ is $\mathsf{U}(2)$-invariant (see, e.g., \cite[Remark 3.2]{Donovan2025}). Such metrics are classified \cite{BerardBergery1982NouvellesEinstein, Derdzinski1983SelfDualKahler, PetersenZhu1995}.
\end{remark}

\subsection{The Einstein equations}
If we set \begin{eqnarray}\label{eqn:notation-Ri-Ric-endomorphism}
    L_i & :=& \frac{f_i'}{f_i}, \qquad R_i : = \frac{f_i}{f_j f_k} \qquad (i,j,k) \ \text{cyclic},
\end{eqnarray} then the mean curvature $S$ of the principal orbit is given by $S = L_1+L_2+L_3$, and the Einstein equation $\text{Ric}(\g)=\lambda \g$ is then the following system of equations, \begin{eqnarray}
     L_i' &=&  -SL_i  + 2R_i^2 -2(R_j-R_k)^2 - \lambda, \label{eqn:L-derivative} \\
     R_i' &=& R_i(L_i - L_j - L_k), \label{eqn:R-derivative}
\end{eqnarray} together with \begin{eqnarray}\label{eqn:lambda-constraint}
    \lambda &=& - \sum_i R_i^2 + 2 \sum_{i<j} R_i R_j - \sum_{i<j} L_i L_j.
\end{eqnarray} 

\noindent The condition $f_i \equiv f_j$ is invariant under the Einstein system of equations and under homothetic rescaling of the metric. Hence, $u_{ij}: = \log(f_i/f_j)$ is the natural scale-invariant coordinate (cf. \cite{Bohm1998InhomogeneousEinstein}). Then \begin{eqnarray*}\label{eqn:alg-eqns-for-u}
    u_{ij}' \ = \ L_i - L_j, \qquad u_{ij} \ = \ -u_{ji}, \qquad u_{ij}+u_{jk} + u_{ki} \ = \ 0,
\end{eqnarray*} and hence, \begin{eqnarray}\label{eqn:main-eqns-for-u}
    u_{ij}'' + S u_{ij}' &=&  4 \frac{(f_i^2 - f_j^2)(f_i^2+f_j^2-f_k^2)}{f_i^2 f_j^2 f_k^2}.
\end{eqnarray}  
We will use the standard smoothness conditions for invariant cohomogeneity-one metrics near a singular orbit. These are treated in \cite{EschenburgWang2000InitialValue, VerdianiZiller2022, Donovan2025}.

\subsection{Self-dual and anti-self-dual $2$-forms}
Consider the g-orthonormal frame $e^0 : = \text{d}t$, $e^i := f_i \sigma_i$. Then \eqref{eqn:structure-equations-sigma} implies \begin{eqnarray}\label{eqn:basic-cartan}
    \text{d}e^i &=& \frac{f_i'}{f_i} e^0 \wedge e^i - \frac{2f_i}{f_j f_k} e^j \wedge e^k \ = \ L_i e^0 \wedge e^i - 2 R_i e^j \wedge e^k.
\end{eqnarray} If $\omega_{ab}$ denote the Levi-Civita connection forms, then $\text{d}e^a = - \omega_{ab} \wedge e^b$, and hence, \begin{eqnarray*}
    \omega_{0i} \ = \ - L_i e^i, \qquad \omega_{jk} \ = \ (R_i - R_j - R_k) e^i.
\end{eqnarray*} The Levi-Civita connection induces a connection on the bundle of self-dual $2$-forms $\Lambda_+^2$ and the bundle of anti-self-dual $2$-forms $\Lambda_-^2$. The standard local frames for the rank $3$ bundles $\Lambda_{\pm}^2$ are given by $$\Omega_i^{\pm} \ = \ e^0 \wedge e^i \pm e^j \wedge e^k,$$ for cyclic $(i,j,k)$. The connection forms $\alpha_i^+ = - A_i e^i$ on $\Lambda_+^2$ are given by \begin{eqnarray*}
    \alpha_i^+ \ = \ -A_i e^i &=& \omega_{0i} + \omega_{jk} \ = \ - L_i e^i + (R_i -R_j - R_k) e^i,
\end{eqnarray*} and therefore, \eqref{eqn:notation-Ri-Ric-endomorphism} yields \begin{eqnarray}\label{eqn:SD-coefficieints}
    A_i & = & L_i + R_j + R_k - R_i \ = \ \frac{f_i'}{f_i} + \frac{f_j^2+f_k^2-f_i^2}{f_i f_j f_k}.  
\end{eqnarray} Similarly, the connection forms $\alpha_i^- = - B_i e^i$ on $\Lambda_-^2$ are given by \begin{eqnarray*}
    \alpha_i^- \ = \ - B_i e^i &=& \omega_{0i} - \omega_{jk} \ = \ - L_i e^i - (R_i - R_j - R_k) e^i,
\end{eqnarray*} and therefore, \eqref{eqn:notation-Ri-Ric-endomorphism} yields  \begin{eqnarray}\label{eqn:ASD-coefficients}
    B_i & = & L_i - R_j - R_k + R_i \ = \ \frac{f_i'}{f_i} + \frac{f_i^2 - f_j^2-f_k^2}{ f_i f_j f_k}.
\end{eqnarray} Then for cyclic $(i,j,k)$, \begin{eqnarray}\label{eqn:d-Omega}
    \text{d}\Omega_i^{+} &=& (A_j + A_k) e^0 \wedge e^j \wedge e^k, \qquad \text{d} \Omega_i^- \ = \ -(B_j + B_k) e^0 \wedge e^j \wedge e^k.
\end{eqnarray} Moreover, \begin{eqnarray}
\nabla \Omega_i^+ &=& - A_j e^j  \otimes \Omega_k^+ - A_k e^k  \otimes \Omega_j^+, \label{eqn:d-Omega+} \\
 \nabla \Omega_i^- &=& -B_k e^k \otimes \Omega_j^- - B_j e^j  \otimes \Omega_k^-. \label{eqn:d-Omega-}
\end{eqnarray} 

\noindent We will make use of the following in Theorem~\ref{thm:SO(3)-CP2} and Theorem~\ref{thm:SO(3)-S2xS2}.

\begin{lemma}\label{lem:SD-Kahler}
    Let $(M^4,\g)$ be an oriented connected Riemannian $4$-manifold with $\Omega \in \Omega^2_{\pm}(M)$ a non-zero self-dual or anti-self-dual $2$-form. If $\Omega$ is parallel with respect to the Levi-Civita connection of g, then g is a K\"ahler metric. 
\end{lemma} \begin{proof}
    Let $\Omega \in \Omega_{\pm}^2(M)$ and define $\mathcal{Q} \in \text{End}(TM)$ by $\g(\mathcal{Q}u,v) = \Omega(u,v)$. If $\Omega$ is parallel, then since $\nabla \g =0$, the endomorphism $\mathcal{Q}$ is parallel. Choose a local oriented orthonormal coframe $e^1, ..., e^4$. A self-dual $2$-form can be written as $$\Omega = a(e^{12} + e^{34}) + b(e^{13}-e^{24})+c(e^{14}+e^{23}), \qquad e^{ij} : = e^i \wedge e^j.$$ The three corresponding endomorphisms form a quaternionic triple $\text{I}, \text{J}, \text{K}$: each squares to $-\text{Id}$, and the three are pairwise anti-commutative. Hence, $\mathcal{Q} = a \text{I} + b \text{J} + c \text{K}$ satisfies $\mathcal{Q}^2 = - (a^2+b^2+c^2)\text{Id}$. Since $| \Omega |^2 = 2(a^2+b^2+c^2)$, it follows that $\mathcal{Q}^2 = - \frac{| \Omega |^2}{2} \text{Id}$. The anti-self-dual case is the same, using the basis $e^{12}-e^{34}$, $e^{13}+e^{24}$, $e^{14}-e^{23}$. Since $\Omega \neq 0$ and $\nabla \Omega =0$, it has constant norm. Set $\text{J}_{\Omega} : = \frac{\sqrt{2}}{| \Omega|} \mathcal{Q}$. Then $\text{J}_{\Omega}^2 = - \text{Id}$, and $\text{J}_{\Omega}$ is skew-symmetric with respect to g. Hence, $\g(\text{J}_{\Omega}u, \text{J}_{\Omega}v) =  \g(u,v)$, and so $\text{J}_{\Omega}$ is a g-orthogonal almost complex structure. Since $\nabla \text{J}_{\Omega}=0$, this almost complex structure is integrable, and g is a K\"ahler metric.
\end{proof}

\noindent The coefficients $A_i$ and $B_i$ satisfy \begin{eqnarray}\label{eqn:Ai+Bi-systems}
    A_i' &=& (R_j + R_k - 3R_i) A_i - A_i^2 + A_j A_k, \\
    B_i' &=& (3R_i - R_j - R_k)B_i - B_i^2 + B_j B_k. \label{eqn:Bi-systems-fix}
\end{eqnarray} Let $\mathfrak{R} : \Lambda^2 \to \Lambda^2$ be the curvature operator. In dimension 4, the curvature operator has the standard decomposition \begin{eqnarray*}
    \mathfrak{R} &=& \begin{pmatrix}
        W^+ + \frac{s}{12} \text{Id}_{\Lambda_+^2} & \text{Ric}_{\circ}^{\sharp} \\
        \text{Ric}_{\circ}^{\sharp} & W^- + \frac{s}{12} \text{Id}_{\Lambda_-^2}
    \end{pmatrix}.
\end{eqnarray*} The Einstein condition implies $\text{Ric}_{\circ} =0$, and hence, $\mathfrak{R}(\Lambda_+^2) \subseteq \Lambda_+^2$ and $\mathfrak{R}(\Lambda_-^2) \subseteq \Lambda_-^2$. Write $F_i^+:=\mathfrak{R}(\Omega_i^+)$ for the curvature of $\Lambda_+^2$. Then  \begin{eqnarray*}
    F_i^+ \ = \ \text{d}\alpha_i^+ - \alpha_j^+ \wedge \alpha_k^+ &=& \text{d}(-A_i e^i) - (-A_j e^j) \wedge (-A_k e^k) \\
    &=& - A_i' e^0 \wedge e^i- A_i \text{d}e^i - A_j A_k e^j \wedge e^k \\
    &\overset{\eqref{eqn:basic-cartan}}{=}& (- A_i' - A_i L_i) e^0 \wedge e^i + (2 A_i R_i - A_j A_k) e^j \wedge e^k \\
    & \overset{\eqref{eqn:SD-coefficieints} \& \eqref{eqn:Ai+Bi-systems}}{=} & (2R_i A_i - A_j A_k) (e^0 \wedge e^i + e^j \wedge e^k).
\end{eqnarray*} Similarly, writing $F_i^- : = \mathfrak{R}(\Omega_i^-)$ for the curvature of $\Lambda_-^2$, then \begin{eqnarray*}
    F_i^- \ = \ \text{d} \alpha_i^- + \alpha_j^- \wedge \alpha_k^- &=& \text{d}(-B_i e^i) + (-B_j e^j) \wedge (-B_k e^k) \\
    &=& (- B_i' - B_i L_i) e^0 \wedge e^i + (2 B_i R_i + B_j B_k) e^j \wedge e^k \\
    &\overset{\eqref{eqn:ASD-coefficients} \& \eqref{eqn:Bi-systems-fix}}{=} & (-2R_i B_i - B_j B_k) (e^0 \wedge e^i - e^j \wedge e^k).
\end{eqnarray*} We write $F_i^+ = \mathfrak{R}(\Omega_i^+) = a_i \Omega_i^+$ and $F_i^- = \mathfrak{R}(\Omega_i^-) = b_i \Omega_i^-$, where \begin{eqnarray}\label{def-ai-bi-eigenvalues}
    a_i \ := \ 2R_i A_i - A_j A_k, \qquad b_i = -2R_i B_i - B_j B_k.
\end{eqnarray}  The Bianchi identity then yields \begin{eqnarray*}
    0 &=& \text{d}F_i^+ - \alpha_j^+ \wedge F_k^+ + \alpha_k^+ \wedge F_j^+ \\
    &=& \text{d}(a_i \Omega_i^+) + A_j e^j \wedge a_k \Omega_k^+ - A_k e^k \wedge a_j \Omega_j^+ \\
    &\overset{\eqref{eqn:d-Omega}}{=} & (a_i' + a_i (A_j + A_k)- a_k A_j - a_j A_k) e^0 \wedge e^j \wedge e^k \\
    &=& (a_i' + A_j(a_i -a_k) + A_k(a_i-a_j)) e^0 \wedge e^j \wedge e^k.
\end{eqnarray*} Similarly, \begin{eqnarray*}
    0 &=& \text{d}F_i^- + \alpha_j^- \wedge F_k^- - \alpha_k^- \wedge F_j^- \\
    &=& \text{d}(b_i \Omega_i^-) - B_je^j \wedge (b_k \Omega_k^-) + B_k e^k \wedge b_j \Omega_j^- \\
    & \overset{\eqref{eqn:d-Omega}}{=} & (-b_i'-b_i(B_j +B_k) + b_kB_j + b_j B_k) e^0 \wedge e^j \wedge e^k.
\end{eqnarray*} Hence, the eigenvalues of $\mathfrak{R} \vert_{\Lambda_{\pm}^2}$ satisfy the system of equations \begin{eqnarray}
    a_i' &=& - A_j(a_i - a_k) - A_k(a_i-a_j), \label{eqn:ai-ode} \\
    b_i' &=& -B_j(b_i-b_k) - B_k(b_i-b_j). \label{eqn:bi-ode}
\end{eqnarray} Moreover, since $\text{Ric}(\g) = \lambda \g$, each block has trace $\text{tr}(\mathcal{W}^{\pm} + \frac{s}{12}\text{Id}_{\Lambda_{\pm}^2})= \frac{3s}{12} = \lambda$. Hence, \begin{eqnarray*}
    \sum_i a_i \ = \ \sum_i b_i \ = \ \lambda.
\end{eqnarray*}

\noindent We will make repeated use of the following basic result (see, e.g., \cite{Bohm1998InhomogeneousEinstein}).

\begin{lemma}\label{lem:regular-singular-germ}
Let $\textbf{X}(t)$ be a smooth $\mathbf{R}^m$-valued germ at $t=0$, satisfying \begin{eqnarray}\label{eqn:ODE-FORM}
    t \textbf{X}'(t) &=& \mathcal{Q} \textbf{X}(t) + t \textbf{B}(t) \textbf{X}(t), 
\end{eqnarray} on $0 < t< \varepsilon$, where the coefficients of the \emph{indicial matrix} $\mathcal{Q}$ are constant and the coefficients of $\textbf{B}(t)$ are bounded. If $\textbf{X}(t) = t^N v + O(t^{N+1})$, with $v \neq 0$ the first non-zero Taylor coefficient of $\textbf{X}(t)$, then $\mathcal{Q}v = N v$. If $\textbf{X}^{(j)}(0)=0$ for all $j \geq 0$, then $\textbf{X}(t) \equiv 0$ for all $t>0$ sufficiently small.
\end{lemma}  \begin{proof}
Let \(N\) be the order of vanishing of \(\mathbf X(t)\) at \(0\), so that $\mathbf X(t)=t^Nv+O(t^{N+1})$ for $v \neq 0$. Then $t\mathbf X'(t)=Nt^Nv+O(t^{N+1})$. Since $\textbf{B}(t)$ is bounded, $\mathcal{Q} \textbf{X}(t)+t\textbf{B}(t) \textbf{X}(t) = t^N \mathcal{Q} v + O(t^{N+1})$. Comparing the coefficients of $t^N$ implies $\mathcal{Q}v = Nv$. Fix $0 < \delta < \varepsilon$.  Then for $0 < t < \delta$, \begin{eqnarray*}
    \frac{\text{d}}{\text{d}t} | \textbf{X}(t) |^2 &=& 2 \left \langle \textbf{X}(t), \textbf{X}'(t) \right \rangle \ \leq \ \left( \frac{2 \| \mathcal{Q} \|}{t}+2\sup_{0 < t < \delta} \| \textbf{B}(t) \| \right) | \mathbf{X}(t) |^2.
\end{eqnarray*} Gronwall's inequality on $[t,\delta]$ gives \begin{eqnarray*}
    |\mathbf X(\delta)|^2 & \leq & |\mathbf X(t)|^2\left(\frac{\delta}{t}\right)^{2q}e^{2\delta \sup_{0 < t < \delta} \| \textbf{B}(t) \|}.
\end{eqnarray*} Since $\textbf{X}^{(j)}(0)=0$ for all $j \geq 0$, for every $k$ we have $| \textbf{X}(t) | = O(t^k)$. Taking $k>q$ and letting $t \to 0^+$ in the preceding estimate implies $| \textbf{X}(\delta) |=0$. Hence, $\textbf{X}(t) \equiv 0$ for all $t>0$ sufficiently small.
\end{proof}

\section{\textbf{The $\mathsf{SU}(2)$ and $\mathsf{SO}(3)$-actions on $S^4$}}\label{sec:3}
\noindent There are three cohomogeneity-one actions on $S^4$. However, the $(\mathsf{SO}(3) \times \mathsf{SO}(2))$-invariant Einstein metrics are homothetic to the round metric \cite{BerardBergery1982NouvellesEinstein}. Hence, we need only consider the $\mathsf{SU}(2)$ and $\mathsf{SO}(3)$-actions.

\subsection{The $\mathsf{SU}(2)$-action on $S^4$}\label{subsec:SU(2)-S4}
Identify $S^4 =\{ (q,t) \in \mathbf{H} \oplus \mathbf{R} : | q|^2 + t^2 =1 \}$ and $\mathsf{SU}(2) \simeq \mathsf{Sp}(1)$. The suspension action is $g \cdot (q,t) = (gq,t)$. For $q \neq 0$, the stabilizer is trivial, so the principal isotropy is $\mathsf{H} = \{ e \}$, and the principal orbits are copies of $\mathsf{SU}(2) / \{ e \} \simeq \mathbf{S}^3$. The singular isotropy groups are $\mathsf{K}_+ = \mathsf{K}_- = \mathsf{SU}(2)$.

 \begin{theorem}\label{thm:SU(2)-S4}
An $\mathsf{SU}(2)$-invariant Einstein metric on $S^4$ is round. 
\end{theorem} \begin{proof}
The endpoint smoothness conditions are the fixed-point case of the smoothness criterion of Verdiani--Ziller \cite[$\S$ 3.1, Eq. (3)]{VerdianiZiller2022}. That is, near $t=0$, and near $s=0$, where $s=T-t$, \begin{eqnarray*}
    f_i^2 &=& t^2 + t^4 \Phi_i(t^2), \qquad f_i^2 = s^2 + s^4 \Psi_i(s^2), \qquad i=1,2,3.
\end{eqnarray*} In particular, the ratios $f_1/f_2$, $f_2/f_3$, and $f_3/f_1$ extend continuously to both endpoints with value $1$. Suppose that $f_1, f_2, f_3$ are not identical. Then \begin{eqnarray*}
    m &: = & \max_{t \in [0,T]} \max_{i,j} u_{ij}(t)
\end{eqnarray*} is positive. Since all ratios equal $1$ at both endpoints, this maximum occurs at some interior point $t_0$. By relabeling if necessary, we may assume that the maximum is realized by $f_1/f_2$. At \(t_0\), since $f_1/f_2$ is maximal, 
\[
    f_1(t_0)\geq f_3(t_0)\geq f_2(t_0),
\]
and since \(m>0\), we have \(f_1(t_0)>f_2(t_0)\). At this interior maximum, \(u'_{12}(t_0)=0\) and \(u''_{12}(t_0)\leq 0\). Applying the difference
equation \eqref{eqn:main-eqns-for-u} to \(u_{12}=\log(f_1/f_2)\), we get
\[
    u''_{12}(t_0)
    =
    4\frac{
    \bigl(f_1(t_0)^2-f_2(t_0)^2\bigr)
    \bigl(f_1(t_0)^2+f_2(t_0)^2-f_3(t_0)^2\bigr)}
    {f_1(t_0)^2f_2(t_0)^2f_3(t_0)^2}.
\]
The first factor in the numerator is positive, while
\[
    f_1(t_0)^2+f_2(t_0)^2-f_3(t_0)^2
    \geq f_2(t_0)^2>0.
\]
Thus \(u''_{12}(t_0)>0\), contradicting \(u''_{12}(t_0)\leq 0\). So $m=0$ and every ratio $f_i / f_j \leq 1$. Therefore, $f_1=f_2=f_3$. The metric is rotationally symmetric, and the Einstein equation reduces to
\[
        \g \ = \ \text{d}t^{2}+f(t)^{2} \g_{\text{rd}},\qquad
        f''+\frac{\lambda}{3}f=0,
\]
with smooth closing conditions \(f(0)=f(T)=0\), \(f'(0)=1\), \(f'(T)=-1\). Hence \(f(t)=\sqrt{3/\lambda}\sin(\sqrt{\lambda/3}\,t)\), so \( \g\) has constant sectional curvature \(\lambda/3\). Thus \(\g\) is round.
\end{proof}

\subsection{The $\mathsf{SO}(3)$-action on $S^4$}\label{subsec:SO(3)-S4}
Let $V : = \text{Sym}_0^2(\mathbf{R}^3)$ be the space of trace-free symmetric matrices with inner product $\langle A,B \rangle := \text{tr}(AB)$. Identify $S^4$ with the unit sphere in $V$. The $\mathsf{SO}(3)$-action on $S^4$ is given by conjugation $g \cdot A : = g A g^{-1}$, $g \in \mathsf{SO}(3)$. The action is effective, has cohomogeneity-one, and the principal isotropy is \begin{eqnarray*}
    \mathsf{H} & = & \{ \text{Id}, \text{diag}(1,-1,-1), \text{diag}(-1,1,-1), \text{diag}(-1,-1,1)\} \ \simeq \ \mathbf{Z}_2\oplus \mathbf{Z}_2.
\end{eqnarray*} At the left endpoint, the singular isotropy is $\mathsf{K}_- = \mathsf{S}(\mathsf{O}(1) \times \mathsf{O}(2)) \subset \mathsf{SO}(3)$, and at the right endpoint, the singular isotropy is $\mathsf{K}_+ = \mathsf{S}(\mathsf{O}(2) \times \mathsf{O}(1)) \subset \mathsf{SO}(3)$. Hence, the singular orbits are $\mathsf{SO}(3)/\mathsf{K}_{\pm} \simeq \mathbf{RP}^2$.

\begin{theorem}\label{thm:SO(3)-S4}
    An $\mathsf{SO}(3)$-invariant Einstein metric on $S^4$ is  round. 
\end{theorem} 

\begin{proof}
Let $\sigma_1$ denote the collapsing direction and $\sigma_2$, $\sigma_3$ the two non-collapsing directions. From \cite[$\S$4]{VerdianiZiller2021}, the smoothness conditions (with $a=4$ in their notation) are given by \begin{eqnarray*} 
    f_1^2 = 16 t^2 + t^4 \Phi_1(t^2), \qquad f_2^2+f_3^2 = \Phi_2(t^2), \qquad f_2^2-f_3^2=t\Phi_3(t^2). 
\end{eqnarray*} Set $h : = f_2(0) = f_3(0)$, $k : = f_1(T)=f_3(T)$, and write $\alpha : = \lambda h^2$. Using \eqref{eqn:notation-Ri-Ric-endomorphism} and \eqref{eqn:SD-coefficieints}, the smoothness conditions imply that \begin{eqnarray}\label{claim-eqn1}
    R_1 = O(t), \qquad R_2, R_3 = \frac{1}{4t}+O(1), \qquad A_1 = \frac{3}{2t} + O(t).
\end{eqnarray}  Then \eqref{eqn:lambda-constraint} and \eqref{eqn:SD-coefficieints} yield \begin{eqnarray}\label{claim-eqn-2}
    A_i &=& \frac{12-\alpha}{2h^2} t +O(t^2), \qquad i=2,3.
\end{eqnarray} Near the right singular orbit, set $s = T-t$, and $\tilde{f}_i(s):=f_i(T-s)$. Then \begin{eqnarray*}
    \tilde{f}_2 = 4s + O(s^3), \quad \tilde{f}_1^2+\tilde{f}_3^2 = \Psi_0(s^2), \quad \tilde{f}_1^2-\tilde{f}_3^2=s\Psi_1(s^2), \quad \Psi_0(0)=2k^2.
\end{eqnarray*} Therefore, for some constant $c$, write $\tilde{f}_1(s)=k+cs+O(s^2)$ and $\tilde{f}_3(s) = k-cs+O(s^2)$, where the opposite sign in the linear term comes from $\tilde{f}_1^2+\tilde{f}_3^2$ being even and $\tilde{f}_1^2-\tilde{f}_3^2$ being odd. Then $L_2 = -\tilde{f}_2'/\tilde{f}_2=-1/s+O(s)$, while $L_1, L_3=O(1)$. Moreover, from \eqref{eqn:notation-Ri-Ric-endomorphism}, $R_2=O(s)$ and \begin{eqnarray*}
    R_1 \ = \ \frac{\tilde{f}_1}{\tilde{f}_2 \tilde{f}_3} \ =  \ \frac{1}{4s} + \frac{c}{2k} + O(s), \qquad R_3  \ = \ \frac{\tilde{f}_3}{\tilde{f}_1 \tilde{f}_2} \ = \ \frac{1}{4s} - \frac{c}{2k} + O(s).
\end{eqnarray*} Hence, from \eqref{eqn:SD-coefficieints}, \begin{eqnarray}\label{eqn-claim-3}
    A_2 \ = \ - \frac{1}{2s} + O(s), \qquad A_i \ = \ O(1), \quad i=1,3,
\end{eqnarray} near the right singular orbit.

\vspace{0.2cm}

\noindent \emph{Claim.} The constant $\alpha=\lambda h^2$ satisfies $\alpha > 12$. \begin{proof}[Proof of Claim.]
Proceed by contradiction and suppose first that $\alpha < 12$. Then \eqref{claim-eqn1} and \eqref{claim-eqn-2} imply that the solution enters the cone $\{ A_1, A_2, A_3 \geq 0 \}$. The cone is invariant under \eqref{eqn:Ai+Bi-systems}. From \eqref{eqn-claim-3}, the right asymptotics are incompatible with a solution trapped in this cone, and therefore, $\alpha \geq 12$. Suppose $\alpha=12$. Then $A_2, A_3=O(t^2)$, and we set $\textbf{X}_A = \begin{pmatrix}
    A_2 \\ A_3
\end{pmatrix}$. From \eqref{eqn:Ai+Bi-systems}, the asymptotics \eqref{claim-eqn1} allow us to write \begin{eqnarray*}
    t \textbf{X}_A'(t) &=& \begin{pmatrix}
    -\frac{1}{2} & \frac{3}{2} \\
    \frac{3}{2} & - \frac{1}{2}
\end{pmatrix} \textbf{X}_A(t) + t \textbf{B}_A(t)\textbf{X}_A(t),
\end{eqnarray*} where the coefficients of $\textbf{B}_A(t)$ are bounded. The eigenvalues of the indicial matrix are $-2$ and $1$. In particular, since $\textbf{X}_A=O(t^2)$, any first non-zero Taylor coefficient would have order $N \geq 2$. The indicial matrix has no eigenvalue $N \geq 2$, so Lemma~\ref{lem:regular-singular-germ} gives $\textbf{X}_A \equiv 0$ near $t=0$. Therefore, near $t=0$, the solution must lie in $\{ A_2 = A_3 = 0 \}$. The common vanishing locus of $A_2, A_3$ is invariant under \eqref{eqn:Ai+Bi-systems}. From \eqref{eqn-claim-3}, the right asymptotics are incompatible with a solution trapped in this common vanishing locus. 
\end{proof}

\noindent From the Claim, $\alpha>12$, so the solution enters the cone $\mathcal C_A:=\{A_1\geq0,\ A_2\leq0,\ A_3\leq0\}$ near $t=0$. This cone is invariant under \eqref{eqn:Ai+Bi-systems}, and hence, the solution remains in $\mathcal{C}_A$. Moreover, for $0 < t <T$, the solution is contained in the interior $\{ A_1>0, A_2<0, A_3<0\}$. To show that the eigenvalues $a_i$ of $\mathfrak{R} \vert_{\Lambda_+^2}$ are all equal, let $F_2 : = a_1-a_2$ and $F_3=a_1-a_3$. The explicit expressions for the $a_i$ given in \eqref{def-ai-bi-eigenvalues}, together with the asymptotics \eqref{claim-eqn1} and  \eqref{claim-eqn-2} imply that \begin{eqnarray*}
    a_1 = \frac{12}{h^2} + O(t^2), \quad a_2, a_3 = \frac{\alpha-12}{2h^2} + O(t), \quad F_2, F_3 = \frac{36-\alpha}{2h^2} + O(t).
\end{eqnarray*} The equations \eqref{eqn:ai-ode} imply \begin{eqnarray}\label{F-ODE-System}
    F_2' &=& (A_1-A_2)F_3-(A_1+2A_3)F_2, \\
    F_3' &=& (A_1-A_3)F_2-(A_1+2A_2)F_3. \nonumber
\end{eqnarray} Set $\mathbf{X}_F : = \begin{pmatrix}
    F_2 \\
    F_3
\end{pmatrix}$. From \eqref{claim-eqn1}--\eqref{claim-eqn-2}, we can write \begin{eqnarray*}
    t \textbf{X}_F'(t) &=& \begin{pmatrix}
        - \frac{3}{2} & \frac{3}{2} \\
        \frac{3}{2} & - \frac{3}{2} 
    \end{pmatrix} \textbf{X}_F(t) + t \textbf{B}_F(t) \textbf{X}_F(t),
\end{eqnarray*} where the coefficients of $\textbf{B}_F(t)$ are bounded. The eigenvalues of the indicial matrix are $-3$ and $0$. Hence, if $\alpha=36$, Lemma~\ref{lem:regular-singular-germ} implies that $F_2=F_3=0$ near $t=0$. The common vanishing locus is preserved by the evolution equations \eqref{F-ODE-System}, and therefore, $a_1=a_2=a_3$ for all $t \in (0,T)$ by uniqueness of the linear system. If $\alpha \neq 36$, set $\varepsilon : = \text{sign}(36-\alpha)$ so that $\varepsilon F_2 >0$ and $\varepsilon F_3>0$ for sufficiently small $t>0$. Since $A_1-A_2>0$ and $A_1-A_3>0$ in $\mathcal{C}_A^{\circ}$, the cones $\{ F_2 \geq 0, F_3 \geq 0\}$ and $\{ F_2 \leq 0, F_3 \leq 0\}$ are invariant under \eqref{F-ODE-System}. Equivalently, the solution is trapped in $\{ \varepsilon F_2 \geq0, \varepsilon F_3 \geq 0 \}$. Near the right singular orbit, the preceding expansions together with \eqref{def-ai-bi-eigenvalues} give constants $c_1,c_2$ such that $a_1 = c_1+O(s)$, $a_2=c_2+O(s)$, $a_3=c_1+O(s)$. In particular, $F_3 = a_1-a_3 = O(s)$. Set $y(s): = \varepsilon F_3(T-s)$. From \eqref{eqn-claim-3}--\eqref{F-ODE-System}, \begin{eqnarray*}
y'(s)-(A_1+2A_2) y(s) &=& -(A_1-A_3) \varepsilon F_2 \ \leq \ 0,
\end{eqnarray*} where the last inequality is a consequence of the solution lying in the cone $\mathcal{C}_A$. Since $-(A_1+2A_2) = \frac{1}{s}+O(1)$, the integrating factor $\mu(s) = s + O(s^2)$. Since $F_3 = a_1-a_3 \to 0$ as $s \to 0$, we have $\mu(s)y(s) \to 0$. Integrating the inequality from $0$ to $s$ gives $\mu(s) y(s) \leq 0$, hence $y \leq 0$. The solution lies in the cone $\{ \varepsilon F_2 \geq 0, \varepsilon F_3 \geq 0\}$, and hence, $y \equiv 0$, and $F_3 \equiv 0$. The evolution equation \eqref{F-ODE-System} for $F_3$ then implies that $0 = (A_1-A_3)F_2$. Since the solution lies in the cone $\mathcal{C}_A$, $A_1-A_3>0$, and $F_2 \equiv 0$.  It follows that $a_1=a_2=a_3 = \frac{\lambda}{3}$, and hence, $\mathfrak{R} \vert_{\Lambda_+^2} = \frac{\lambda}{3} \text{Id}_{\Lambda_+^2}$. Since the metric is Einstein, \begin{eqnarray*}
    \mathfrak{R} \vert_{\Lambda_+^2} \ = \ \mathcal{W}^+ + \frac{s}{12} \text{Id}_{\Lambda_+^2} \ = \ \mathcal{W}^+ + \frac{\lambda}{3} \text{Id}_{\Lambda_+^2} \ = \ \frac{\lambda}{3} \text{Id}_{\Lambda_+^2}.
\end{eqnarray*} Thus, $\mathcal{W}^+=0$ and the metric is anti-self-dual. On $S^4$, the Hirzebruch signature theorem then implies that \begin{eqnarray*}
    0 \ = \ \tau(S^4) \ = \ \frac{1}{12\pi^2} \int_{S^4} ( | \mathcal{W}^+ |^2 - | \mathcal{W}^- |^2) \text{dvol} \ = \ - \frac{1}{12\pi^2} \int_{S^4} | \mathcal{W}^- |^2 \text{dvol}.
\end{eqnarray*} The metric is therefore conformally flat. A conformally flat Einstein metric with $\lambda>0$ has constant sectional curvature $\lambda/3$, and is thus homothetic to the round metric.
\end{proof}

\noindent The uniqueness of the round metric among the cohomogeneity-one Einstein metrics on $S^4$ is surprising in light of the results of \cite{Bohm1998InhomogeneousEinstein, FoscoloHaskins2017, Chi2024, NienhausWink2025TenSphere, ButtsworthHodgkinson2026S12}. Hence, we pose the following question.

\begin{question}
    Are there even-dimensional spheres $S^{2k}$, $k \geq 7$, for which the round metric is the unique Einstein metric invariant under a cohomogeneity-one action?
\end{question}

\section{\textbf{The $\mathsf{SU}(2)$ and $\mathsf{SO}(3)$-actions on $\mathbf{CP}^2$}}

\noindent There are two cohomogeneity-one actions on $\mathbf{CP}^2$, an action of $\mathsf{SU}(2)$ and an action of $\mathsf{SO}(3)$.

\subsection{The $\mathsf{SU}(2)$-action on $\mathbf{CP}^2$}\label{subsec:SU(2)-CP2}
Write $\mathbf{CP}^2 = \mathbf{P}(\mathbf{C} \oplus \mathbf{C}^2)$, and let $\mathsf{SU}(2)$ act trivially on the $\mathbf{C}$-factor and by the standard representation on $\mathbf{C}^2$, i.e., $g \cdot [z_0 : v] = [z_0 : gv ]$. For $z_0 \neq 0$ and $v \neq 0$, the stabilizer is trivial, so $\mathsf{H} = \{ e \}$. Since $[1:0] \in \mathbf{P}(\mathbf{C} \oplus \mathbf{C}^2)$ is fixed, the singular isotropy $\mathsf{K}_- = \mathsf{SU}(2)$. The hyperplane $\mathbf{P}(\mathbf{C}^2) \simeq \mathbf{CP}^1 \simeq S^2$ is a singular orbit, with isotropy $\mathsf{K}_+ = S^1 \subset \mathsf{SU}(2)$.

\begin{theorem}\label{thm:SU(2)-CP2}
    An $\mathsf{SU}(2)$-invariant Einstein metric on $\mathbf{CP}^2$ is $\mathsf{U}(2)$-invariant, and thus, homothetic to the Fubini--Study metric.
\end{theorem} \begin{proof}
From the smoothness conditions \cite[§3.1, (8), (9), Rem.~3.5, Table~4, Lem.~3.6]{Donovan2025}, which specialize the general criterion of Verdiani--Ziller \cite[Thm.~1, Prop.~4]{VerdianiZiller2022}, the fixed-point end is modeled on the standard \(\mathsf{SU}(2)\)-action on \(\mathbf C^2\), and hence, after using arc-length \(t\), $f_i^2=t^2+t^4\Phi_i(t^2)$, $i=1,2,3$. At the \(\mathbf{CP}^1\)-end, with \(s=T-t\) and with \(\sigma_3\) the collapsing direction, the \(n=1\) case of the \(O(-n)\)-smoothness condition gives $f_3^2=s^2+s^4\Phi_3(s^2)$, $f_1^2+f_2^2=\Psi_0(s^2)$, and $f_1^2-f_2^2=s^4\Psi_1(s^2)$, with $\Psi_0(0)>0$.  In particular, for all $i,j$, the ratios $f_i/f_j \to 1$ as $t \to 0$, while \begin{eqnarray*}
    \lim_{t \to T} f_1/f_2 \ = \ \lim_{t \to T} f_2/f_1 \ = \ 1, \qquad \lim_{t \to T} f_3/f_1 \ = \ \lim_{t \to T} f_3/f_2 \ = \ 0.
\end{eqnarray*} Hence, the ratios $f_1/f_2$, $f_2/f_1$, $f_3/f_1$, $f_3/f_2$ extend continuously to $[0,T]$ with boundary values at most $1$. We will show that each of these ratios is bounded above by $1$ on $[0,T]$. Proceed by contradiction and suppose this is not the case. Let $\Gamma : = \{ (1,2), (2,1), (3,1),(3,2)\}$. Then for some $(i,j) \in \Gamma$, and some $t_0 \in (0,T)$, the ratio $f_i/f_j$ achieves an interior maximum at $t_0$ with $f_i(t_0)/f_j(t_0) > 1$. Let $k$ be the remaining index. Observe that if $(i,j) \in \Gamma$ and $k \neq i$, $k \neq j$, then $(k,j) \in \Gamma$. Hence, $f_i(t_0) > f_j(t_0)$, and $f_i(t_0) \geq f_k(t_0)$, and \eqref{eqn:main-eqns-for-u} implies that \begin{eqnarray*}
    0 \ \geq \ u_{ij}''(t_0) & = & 4 \frac{(f_i^2 - f_j^2)(f_i^2+f_j^2-f_k^2)}{f_i^2 f_j^2 f_k^2} \ > \ 0.
\end{eqnarray*} It follows that $f_1/f_2 \leq 1$ and $f_2/f_1 \leq 1$, and thus, $f_1 \equiv f_2$. From Remark~\ref{remark:2.4}, the metric g is $\mathsf{U}(2)$-invariant. 
\end{proof}

\subsection{The $\mathsf{SO}(3)$-action on $\mathbf{CP}^2$}\label{subsec:SO(3)-CP2}
The $\mathsf{SO}(3)$-action on $\mathbf{CP}^2$ is obtained from the transitive action of $\mathsf{SU}(3)$. Indeed, write  $\mathbf{CP}^2 = \mathbf{P}(\mathbf{C}^3)$, and let $\mathsf{SO}(3) \subset \mathsf{SU}(3)$ act by real matrices, complex linearly. That is, for $A \in \mathsf{SO}(3)$ and $z \in \mathbf{C}^3$, we have $A \cdot [z] = [Az]$.  The singular orbits are a totally real $\mathbf{RP}^2 \subset \mathbf{CP}^2$ with isotropy $\mathsf{K}_- = \mathsf{S}(\mathsf{O}(1) \times \mathsf{O}(2)) \simeq \mathsf{O}(2)$, and the conic $\{ [z_0:z_1:z_2]: z_0^2+z_1^2+z_2^2=0 \} \simeq \mathbf{S}^2$, with isotropy $\mathsf{K}_+ = \mathsf{SO}(2)$.

\begin{theorem}\label{thm:SO(3)-CP2}
    An $\mathsf{SO}(3)$-invariant Einstein metric on $\mathbf{CP}^2$ is K\"ahler, and thus, homothetic to the Fubini--Study metric.
\end{theorem}

\begin{proof}
Put the singular orbit with isotropy $\mathsf K_-=\mathsf S(\mathsf O(1)\times \mathsf O(2))$ at \(t=0\), and the singular orbit with isotropy
$\mathsf K_+=\mathsf{SO}(2)$ at $t=T$. We choose the invariant coframe so that the \(\sigma_1\)-direction collapses at \(t=0\), and the \(\sigma_3\)-direction collapses at \(t=T\). From \cite[Proposition 3.1]{BettiolKrishnan2019}, the smoothness conditions near \(t=0\) are\begin{eqnarray*}
    f_1^2 = 4t^2+t^4 \Phi_1(t^2), \quad f_2^2 +f_3^2=\Phi_0(t^2), \quad f_2^2-f_3^2=t^2\Phi_2(t^2),
\end{eqnarray*} where $\Phi_0(0)>0$. Let $s := T-t$ and $\tilde{f}_i(s) := f_i(T-s)$. The smoothness conditions near $t=T$ are \begin{eqnarray*}
    \tilde{f}_3^2=16s^2+s^4\Psi_3(s^2), \quad  \tilde{f}_1^2+ \tilde{f}_2^2=\Psi_0(s^2), \quad  \tilde{f}_1^2-\tilde{f}_2^2=s\Psi_1(s^2),
\end{eqnarray*} where $\Psi_0(0)>0$. We will show that $\Omega_3^-$ is a parallel anti-self-dual $2$-form. From \eqref{eqn:d-Omega-}, the $2$-form $\Omega_3^-$ is parallel if and only if $B_1=B_2=0$. Let $\mathcal{K} : = \{ B_1= B_2 = 0 \}$. Then $\mathcal{K}$ is invariant under the evolution equations \eqref{eqn:Bi-systems-fix}. From \eqref{eqn:notation-Ri-Ric-endomorphism}, \eqref{eqn:L-derivative}, \eqref{eqn:lambda-constraint}, and \eqref{eqn:ASD-coefficients}, the smoothness conditions near $t=0$ imply that \begin{eqnarray}\label{eqn-B1-claim-1}
    B_1=O(t^3), \qquad R_2+R_3-3R_1+B_1=\frac1t+O(t).
\end{eqnarray} Set $q : = f_1(T) = f_2(T)>0$ and $\beta : = \lambda q^2$. Set $\widetilde{B}_i(s) : = B_i(T-s)$. The smoothness conditions near $t=T$, together with \eqref{eqn:ASD-coefficients} and \eqref{eqn:lambda-constraint} yield \begin{eqnarray}\label{eqn:tilde-Bi}
\widetilde{B}_3 =-\frac{3}{2s}+O(s), \qquad \widetilde{B}_i =\frac{\beta-12}{2q^2}s+O(s^2),\qquad i=1,2.
\end{eqnarray}

\vspace{0.2cm}

\noindent \emph{Claim.} The constant $\beta = 12$. \begin{proof}[Proof of Claim.]
    Proceed by contradiction and suppose that $\beta \neq 12$. Set $\varepsilon : = \text{sign}(\beta-12)$. For $s>0$ sufficiently small, the solution lies in the cone $\mathcal C_\varepsilon := \{\varepsilon B_1\ge0,\ \varepsilon B_2\ge0,\ B_3\le0\}$. This cone is backward invariant under \eqref{eqn:Bi-systems-fix}. Moreover, the inequalities are strict on compact subintervals of \((0,T)\). The evolution equation \eqref{eqn:Bi-systems-fix} for $B_1$ can be written as \begin{eqnarray*}
    B_1' + (R_2 + R_3 - 3R_1 + B_1)B_1  &=& B_2 B_3.
\end{eqnarray*} From \eqref{eqn-B1-claim-1}, the integrating factor $\mu(t)$ for this equation satisfies $\mu(t) = t + O(t^2)$. Since the solution lies in the cone $\mathcal{C}_{\varepsilon}$, we know that $\varepsilon B_1 \geq 0$ and $\varepsilon B_2 B_3 \leq 0$. Hence, \begin{eqnarray}\label{eqn-integralb}
    0 \leq \varepsilon B_1(t) \mu(t) &=& \int_0^t \varepsilon \mu(\zeta) B_2(\zeta) B_3(\zeta) d\zeta \ \leq \ 0.
\end{eqnarray} Hence, $\varepsilon B_1(t) \equiv 0$ near $t=0$. Inserting this into \eqref{eqn-integralb} implies $B_2(t) B_3(t) \equiv 0$ near $t=0$. Since the inequalities defining $\mathcal{C}_{\varepsilon}$ are strict on compact subintervals of $(0,T)$, we must have $\beta=12$.
\end{proof}

\noindent Set $\textbf{X}_{\widetilde{B}}(s) : = \begin{pmatrix}
    \widetilde{B}_1(s) \\ \widetilde{B}_2(s)
\end{pmatrix}$. Since $\beta=12$, near $s=0$ we have $\textbf{X}_{\widetilde{B}}(s) = O(s^2)$. Moreover, from \eqref{eqn:tilde-Bi}, the evolution equations \eqref{eqn:Bi-systems-fix} can be written as \begin{eqnarray*}
    s \textbf{X}_{\widetilde{B}}'(s) &=& \begin{pmatrix}
        -\frac12&\frac32\\
        \frac32&-\frac12
    \end{pmatrix} \textbf{X}_{\widetilde{B}}(s) + s \textbf{B}_{\widetilde{B}}(s) \textbf{X}_{\widetilde{B}}(s),
\end{eqnarray*} where the coefficients of $\textbf{B}_{\widetilde{B}}(s)$ are bounded. Since the eigenvalues of the indicial matrix are $-2$ and $1$, Lemma~\ref{lem:regular-singular-germ} implies $\widetilde{B}_1 = \widetilde{B}_2 =0$ near $s=0$. Hence, by the invariance of $\mathcal{K} = \{ B_1 = B_2 = 0 \}$, this holds on the entire regular interval. 

\vspace{0.2cm}

\noindent \emph{Claim.} The $2$-form $\Omega_3^-$ extends uniquely to a smooth parallel $2$-form on all of $M$.

\begin{proof}[Proof of Claim.]
    Let $S = S_- \cup S_+$ be the union of the singular orbits. On $M \setminus S$, the $2$-form $\Omega_3^-$ is smooth and satisfies $| \Omega_3^- |_{\g}^2 = 2$. In particular, $\Omega_3^-$ is locally bounded near $S$. Let $U$ be a local chart with coordinates $(y,z) \in \mathbf{R}^2 \times \mathbf{R}^2$ such that $U \cap S = \{ z = 0 \}$. We choose $U$ sufficiently small so that $\Lambda^2_-$ is smoothly trivialized over $U$ with local frame $e_{\alpha}$. Writing $\Omega_3^- = \sum_{\alpha} u^{\alpha} e_{\alpha}$, the equation $\nabla \Omega_3^- =0$ is equivalent to $\partial_k u^{\alpha} + \Gamma_{k\beta}^{\alpha} u^{\beta}=0$ on $U \cap \{ z \neq 0 \}$, where $\Gamma_{k \beta}^{\alpha}$ are smooth functions. Since $S$ has codimension $2$, extend each $u^{\alpha} \in \mathcal{C}^{\infty}(U \cap \{ z \neq 0 \})$ trivially to $L^{\infty}(U)$. Let $D_{\varepsilon} := \{ | z | > \varepsilon \}$. For $\varphi \in \mathcal{C}^{\infty}_c(U)$, \begin{eqnarray*}
        \int_{D_{\varepsilon}} u^{\alpha} \partial_k \varphi   &=& \int_{\partial D_{\varepsilon}} u^{\alpha} \varphi \nu_k + \int_{D_{\varepsilon}}  \Gamma_{k\beta}^{\alpha} u^{\beta} \varphi .
    \end{eqnarray*} Since $S$ has codimension $2$, the area of the boundary of an $\varepsilon$-tube around $S$ is $O(\varepsilon)$. Therefore, the first integral converges to zero as $\varepsilon \searrow 0$. Since $\Gamma_{k\beta}^{\alpha}$, $u^{\beta}$ and $\varphi$ are bounded, the second integral converges by the dominated convergence theorem. Hence, $\partial_k u^{\alpha} + \Gamma_{k\beta}^{\alpha} u^{\beta}=0$ holds on $U$ in the sense of distributions. Therefore, $u \in W_{\text{loc}}^{1,\infty}$ and bootstrapping with this equation implies that $u \in W_{\text{loc}}^{k, \infty} \subset \mathcal{C}^{k-1}_{\text{loc}}$ for all $k \in \mathbf{N}$.
\end{proof}

\noindent From the claim, $\Omega_3^-$ defines a smooth parallel anti-self-dual $2$-form on $M$. From Lemma~\ref{lem:SD-Kahler}, the metric g is K\"ahler--Einstein. The only complex surface admitting a K\"ahler--Einstein metric with $\lambda>0$ and diffeomorphic to $\mathbf{CP}^2$ is $\mathbf{CP}^2$, endowed with the Fubini--Study metric (see \cite{Tian1987KahlerEinsteinCertain, Tian1990CalabiConjectureSurfaces} together with \cite{barth2004compact} for the general classification). Since the K\"ahler--Einstein metric is unique up to homothety \cite{BandoMabuchi1987}, this completes the proof.  
\end{proof}

\section{\textbf{The $\mathsf{SU}(2)$-action on $\mathbf{CP}^2 \sharp \overline{\mathbf{CP}}^2$}}\label{sec:Trapping}
\noindent We now consider the cohomogeneity-one actions on $M:= \mathbf{CP}^2 \sharp \overline{\mathbf{CP}}^2$, the blow-up $\pi : M \to \mathbf{CP}^2$ of $\mathbf{CP}^2$ at one point $p \in \mathbf{CP}^2$. We can write $M = \{ ([w:z], [\ell]) \in \mathbf{P}(\mathbf{C} \oplus \mathbf{C}^2) \times \mathbf{P}(\mathbf{C}^2) : z \in \ell \}$, with the $\mathsf{SU}(2)$-action given by $g \cdot ([w:z],[\ell]) = ([w:gz],[g\ell])$. One singular orbit is the exceptional divisor $E:= \pi^{-1}(p) \simeq \mathbf{CP}^1 \simeq \mathsf{SU}(2)/\mathsf{U}(1)$, while the other is $\{ w=0\} \simeq \mathbf{CP}^1 \simeq \mathsf{SU}(2)/\mathsf{U}(1)$, the strict transform of the projective line not passing through $p$.

\begin{theorem}\label{thm:SU(2)-CP2-blow-up}
    An $\mathsf{SU}(2)$-invariant Einstein metric on $\mathbf{CP}^2 \sharp \overline{\mathbf{CP}}^2$ is homothetic to the Page metric.
\end{theorem}  \begin{proof}
 There is a family of $\mathsf{SU}(2)$ cohomogeneity-one actions on $\mathbf{CP}^2 \sharp \overline{\mathbf{CP}}^2$ with principal isotropy $\mathbf{Z}_n$ for $n>1$ odd, or trivial isotropy \cite{Parker1986, Hoelscher2010, GroveZiller2011Lifting}. From \cite[Remark 1.3]{Donovan2025}, if the isotropy is non-trivial, then the $\mathsf{SU}(2)$-invariant Einstein metric is $\mathsf{U}(2)$-invariant. Hence, from \cite{BerardBergery1982NouvellesEinstein, Derdzinski1983SelfDualKahler}, the metric is homothetic to the Page metric. So we may assume the principal isotropy is trivial $\mathsf{H} = \{ e \}$. From Lemma~\ref{lem:diagonal}, we may assume the metric is of the form \eqref{eqn:metric-form}.

Let \( \mathsf{K}_\pm\) be the two circle singular isotropy groups obtained from an admissible transverse curve $c \subset M$ with \(c(0):=p_-\) and $c(T) : = p_+$. Since $\mathsf{N}_{\mathsf{SU}(2)}(\mathsf{H})^\circ=\mathsf{SU}(2)$, and since all circle subgroups of \(\mathsf{SU}(2)\) are conjugate, choose $a \in \mathsf{N}_{\mathsf{SU}(2)}(\mathsf{H})^{\circ}$ such that $a \mathsf{K}_+ a^{-1} = \mathsf{K}_-$. From \cite[Remark 2.6]{VerdianiZiller2022}, changing the admissible transverse curve inside the component of $M$ containing $p_-$, from $p_-$ to $ap_+$, changes the representative of the second singular isotropy group from \( \mathsf{K}_+\) to \(a \mathsf{K}_+a^{-1}\), without changing the resulting class of
\(\mathsf{G}\)-invariant metrics up to \(\mathsf{G}\)-equivariant isometry. Hence, for the classification up to isometry, we may assume that the group diagram is $\{e\}\subset \{\mathsf{K}_-,\mathsf{K}_-\}\subset \mathsf{SU}(2)$.  Choose \(X_1\in\mathfrak{su}(2)\) generating \(\mathfrak k := \text{Lie}(\mathsf{K}_-)\). Under this change of admissible curve we pull back the invariant metric by the corresponding $\mathsf{G}$-equivariant diffeomorphism; the diagonal normal form is therefore preserved after relabeling the diagonal basis, and the collapsing direction at both endpoints may be taken to be $X_1$. The smoothness conditions at both singular orbits then give $f_1 \to 0$, $\frac{f_2}{f_3} \to 1$, and $\frac{f_3}{f_2} \to 1$. Consider the four ratios $\left \{ \frac{f_1}{f_2}, \frac{f_1}{f_3}, \frac{f_2}{f_3}, \frac{f_3}{f_2} \right \}$. At both endpoints, all four ratios have boundary value $\leq 1$. Suppose one of these ratios exceeds $1$. Let $t_0\in(0,T)$ be the interior point where the maximum among these ratios, say $\frac{f_i}{f_j}$, is achieved. Let $k$ be the remaining index. Then at $t_0$, $f_i(t_0) > f_j(t_0)$, and since $\frac{f_i}{f_j}(t_0) \geq \frac{f_k}{f_j}(t_0)$, it follows that $f_i^2> f_j^2$ and $f_i^2+f_j^2-f_k^2 \geq f_j^2 >0$. The maximum of $u_{ij}=\log(f_i/f_j)$ occurs at the same point $t_0 \in (0,T)$. Hence, from \eqref{eqn:main-eqns-for-u} we see that at $t_0$, \begin{eqnarray*}
        0 \ \geq \ u_{ij}''(t_0) \ = \ 4 \frac{(f_i^2 - f_j^2)(f_i^2+f_j^2-f_k^2)}{f_i^2 f_j^2 f_k^2} \ > \ 0,
    \end{eqnarray*}  and hence, $\frac{f_2}{f_3} \leq 1$ and $\frac{f_3}{f_2} \leq 1$. Therefore, $f_2 \equiv f_3$, and Remark~\ref{remark:2.4} implies that the metric is $\mathsf{U}(2)$-invariant. Hence, from \cite{BerardBergery1982NouvellesEinstein, Derdzinski1983SelfDualKahler,PetersenZhu1995}, the metric is homothetic to the Page metric.
\end{proof}

\section{\textbf{The $\mathsf{SO}(3)$-action on $S^2 \times S^2$}}\label{sec:next}
\noindent There are two cohomogeneity-one actions on $S^2 \times S^2$. The cohomogeneity-one action of $\mathsf{SO}(3) \times \mathsf{SO}(2)$ is handled by \cite{BerardBergery1982NouvellesEinstein, Derdzinski1983SelfDualKahler}, so we need only consider the action of $\mathsf{SO}(3)$. The cohomogeneity-one action of $\mathsf{SO}(3)$ on $S^2 \times S^2$ is the diagonal action $A \cdot (x,y) = (Ax,Ay)$, $A \in \mathsf{SO}(3)$, $x,y \in S^2$. The singular orbits are $\{ (x,x) : x \in S^2 \} \simeq S^2$ and $\{ (x,-x) : x \in S^2 \} \simeq S^2$. The principal isotropy is trivial $\mathsf{H} = \{ e \}$, and the singular orbits are both $S^2 \simeq \mathsf{SO}(3)/\mathsf{SO}(2)$.

\begin{theorem}\label{thm:SO(3)-S2xS2}
Let \(\g\) be a smooth \(\mathsf{SO}(3)\)-invariant Einstein metric on \(S^2\times S^2\). Then g is homothetic to the product of the round metrics on each factor. 
\end{theorem}

\begin{proof}
The proof is very similar to the proof of Theorem~\ref{thm:SO(3)-CP2}. The smoothness conditions are specified in \cite[Section 4, Eq. (4.1)]{VerdianiZiller2022}. Writing the metric in the form \eqref{eqn:metric-form}, after relabeling, the \(\sigma_3\)-direction collapses at \(t=0\), while the \(\sigma_1\)-direction collapses at \(t=T\). Thus the smoothness conditions near $t=0$ are \begin{eqnarray*}
    f_3^2 = 4t^2 + t^4 \Phi_3(t^2), \qquad f_1^2+f_2^2 = \Phi_0(t^2), \qquad f_1^2-f_2^2=t^2 \Phi_{12}(t^2),
\end{eqnarray*}  with $\Phi_0(0)=2q^2>0$. Writing $s : = T-t$ and $\tilde{f}_i(s) : = f_i(T-s)$, the smoothness conditions near $s=0$ are \begin{eqnarray*}
    \tilde{f}_1^2 = 4s^2 + s^4\Psi_1(s^2), \qquad \tilde{f}_2^2+\tilde{f}_3^2=\Psi_0(s^2), \qquad \tilde{f}_2^2  - \tilde{f}_3^2 =s^2\Psi_{23}(s^2),
\end{eqnarray*} where $\Psi_0(0)=2Q^2>0$. From \eqref{eqn:d-Omega+}, the self-dual $2$-form $\Omega_3^+$ is parallel if $A_1=A_2=0$. The common vanishing locus $\mathcal{K} = \{ A_1 = A_2 = 0 \}$ is invariant under \eqref{eqn:Ai+Bi-systems}. Set $\delta : = 8 - \lambda q^2$. Then \eqref{eqn:lambda-constraint} and \eqref{eqn:SD-coefficieints} yield \begin{eqnarray*}
    A_3 \ = \ \frac{2}{t} + O(t), \qquad A_i \ = \ \frac{\delta}{2q^2} t + O(t^3), \qquad i=1,2.
\end{eqnarray*}

\noindent The same argument used in the proof of Theorem~\ref{thm:SO(3)-CP2} shows that $\delta =0$. Indeed, if $\delta \neq 0$, we set  \(\varepsilon=\operatorname{sign}(\delta)\). Then near \(t=0\), the solution lies in the cone $\mathcal{C}_{\varepsilon} : = \{ \varepsilon A_1 \geq 0, \varepsilon A_2 \geq 0, A_3 \geq 0 \}$. The cone is invariant under \eqref{eqn:Ai+Bi-systems}, and the inequalities defining $\mathcal{C}_{\varepsilon}$ are strict on compact subintervals of $(0,T)$. Using the smoothness expansion at $s=0$, together with \eqref{eqn:L-derivative}, \eqref{eqn:lambda-constraint}, and \eqref{eqn:Ai+Bi-systems}, we have \begin{eqnarray*}
    A_1(T-s)=O(s^3), \qquad R_2+R_3-3R_1-A_1=\frac{1}{s}+O(s).
\end{eqnarray*} Set \(y(s):=A_1(T-s)\). The evolution equation for $A_1$ has an integrating factor with $\mu(s) = s + O(s^2)$. The argument for Theorem~\ref{thm:SO(3)-CP2} can be followed almost verbatim to show that $y \equiv 0$, and hence, $A_2 A_3 \equiv 0$ near $s=0$. Since the inequalities are strict on compact subintervals of $(0,T)$, we similarly deduce that $\delta =0$. Let $\textbf{X}_A:=\begin{pmatrix}
        A_1\\
        A_2
    \end{pmatrix}$. Near \(t=0\), write
\[
    t\textbf{X}_A'(t)= \begin{pmatrix}
        -1&2\\
        2&-1
    \end{pmatrix} \textbf{X}_A(t)+t\textbf{B}_A(t) \textbf{X}_A(t),
\] where the coefficients of $\textbf{B}_A$ are bounded. Since $\textbf{X}_A(t) = O(t^3)$ and the eigenvalues of the indicial matrix are $-3$ and $1$, Lemma~\ref{lem:regular-singular-germ} implies that $A_1=A_2=0$ near $t=0$.  By invariance of the locus \(\{A_1=A_2=0\}\), this holds on the whole regular interval. The argument in Theorem~\ref{thm:SO(3)-CP2} shows that $\Omega_3^+$ extends smoothly as a parallel self-dual $2$-form on $M$. Hence, Lemma~\ref{lem:SD-Kahler} implies that g is K\"ahler--Einstein. The only complex surface admitting a K\"ahler--Einstein metric with $\lambda>0$ and diffeomorphic to $S^2 \times S^2$ is $\mathbf{CP}^1 \times \mathbf{CP}^1$, endowed with the product of the Fubini--Study metrics (see \cite{Tian1987KahlerEinsteinCertain, Tian1990CalabiConjectureSurfaces} and \cite{barth2004compact} for the general classification). Since the K\"ahler--Einstein metric is unique up to homothety \cite{BandoMabuchi1987}, this completes the proof. 

\end{proof}

\section{\textbf{Orbifold Einstein metrics on $S^4$}}

\begin{theorem}
Let $\mathcal{O}_k$ denote the $\mathsf{SO}(3)$-orbifold on $S^4$ with singular orbit $B_+ \simeq \mathbf{RP}^2$ and normal angle $2\pi/k$. If g is an $\mathsf{SO}(3)$-invariant Einstein orbifold metric on $\mathcal{O}_k$, then g is homothetic to Hitchin's metric.  
\end{theorem}

\begin{proof}
The case \(k=1\) is Theorem~\ref{thm:SO(3)-S4}. Hence assume \(k \geq 2\). Put the smooth end at \(t=0\), the orbifold end at \(t=T\), and write \(s=T-t\) near the right endpoint. We choose the invariant coframe so that the \(\sigma_1\)-direction collapses at \(t=0\), and the \(\sigma_2\)-direction collapses at \(t=T\). Thus, on the regular set, write the metric in the form \eqref{eqn:metric-form}.  At the left end, the smoothness conditions are the same as in Theorem~\ref{thm:SO(3)-S4}. We will only use
\[
B_1=\frac{1}{2t}+O(t), \qquad B_2,B_3=O(1),
\]
and
\[
b_2-b_3\to 0, \qquad 2B_1+B_2=\frac1t+O(1).
\]

\noindent Set $q:=f_1(T)=f_3(T)>0$ and $\beta:= \lambda q^2$, and $\Theta_k : = 4 + 8/k$.  The orbifold smoothness conditions at the right end give $f_2(T-s)=\frac{4}{k}s+O(s^3)$, while \(f_1(T-s)\) and \(f_3(T-s)\) tend to \(q\). Since \(k\geq 2\), the
singular orbit is totally geodesic in an orbifold chart, and hence \(f_1(T-s)-q\) and \(f_3(T-s)-q\) are \(O(s^2)\). For \(k>2\), the quadratic
terms of \(f_1\) and \(f_3\) agree. Write \( \tilde{f}_i(s):=f_i(T-s)\), and set \(\rho:=4/k\). Then
\[
\tilde{f}_2(s)= \rho s+O(s^3),\qquad
\tilde{f}_1(s)=q+us^2+O(s^3),\qquad
\tilde{f}_3(s)=q+vs^2+O(s^3).
\]
For \(k>2\), smoothness gives \(u=v\). For \(k=2\), the coefficient \(u-v\) may be non-zero. A direct expansion gives
\[
B_2=-\left(1+\frac2\rho\right)\frac1s+O(s),
\]
and
\[
B_1-B_3 \ = \ 
\left(-\frac2q+\frac4{\rho q}\right)(u-v)s+O(s^2) \ = \ 
\frac{k-2}{q}(u-v)s+O(s^2).
\]
Thus the linear terms of \(B_1\) and \(B_3\) agree for all \(k\geq2\). Moreover,
\[
\frac{B_1+B_3}{2}
=
\left(-\frac{u+v}{q}-\frac{\rho}{q^2}\right)s+O(s^2).
\]
From \eqref{eqn:lambda-constraint}, $\lambda= 4/q^2 - \frac{2}{q}(u+v)$.  Hence, with \(\beta=\lambda q^2\) and \(\Theta_k=4+8/k=4+2\rho\), \begin{eqnarray}\label{hitchin-eq1}
    B_2=-\frac{k+2}{2s}+O(s),\qquad B_i=\frac{\beta-\Theta_k}{2q^2}s+O(s^2),\quad i=1,3.
\end{eqnarray} Consequently,
\[
b_2\to \frac{\Theta_k}{q^2},\qquad
b_1,b_3\to \frac{\beta-\Theta_k}{2q^2}.
\]

\vspace{0.2cm}

\noindent \emph{Claim.} The constant $\beta>\Theta_k$. \begin{proof}
    First suppose that $\beta < \Theta_k$. Then for $s>0$ sufficiently small, the solution lies in the cone $\{ B_1 \leq 0, B_2 \leq 0, B_3 \leq 0 \}$. This cone is backward invariant under \eqref{eqn:Bi-systems-fix}. Indeed, on the face \(B_i=0\), one has \(B_i'=B_jB_k\geq 0\). Hence integrating backwards from the right end gives \(B_i\leq 0\) on the whole regular interval. This violates the asymptotics for $B_1 = \frac{1}{2t} + O(t)$ near the left singular orbit. Suppose $\beta = \Theta_k$ and let $\widetilde{B}_1(s) : = B_1(T-s)$ and $\widetilde{B}_3(s) : = B_3(T-s)$. Set $\textbf{X}_{\widetilde{B}}(s) : = \begin{pmatrix}
        \widetilde{B}_1(s) \\ \widetilde{B}_3(s)
    \end{pmatrix}$. From \eqref{hitchin-eq1}, we can write \begin{eqnarray*}
        s \textbf{X}_{\widetilde{B}}'(s) &=& \begin{pmatrix}
-\frac{k}{2} & \frac{k+2}{2}\\
\frac{k+2}{2} & -\frac{k}{2}
\end{pmatrix} \textbf{X}_{\widetilde{B}}(s) + s \textbf{B}_{\widetilde{B}}(s) \textbf{X}_{\widetilde{B}}(s),
    \end{eqnarray*} where the coefficients of $\textbf{B}_{\widetilde{B}}$ are bounded. If $\beta = \Theta_k$, then $\textbf{X}_{\widetilde{B}}(s) = O(s^2)$. Since the eigenvalues of the indicial matrix are $-(k+1)$ and $1$, Lemma~\ref{lem:regular-singular-germ} implies $\textbf{X}_{\widetilde{B}}(s) \equiv 0$ for all $s>0$ sufficiently small. The common vanishing locus $\{ B_1 = B_3 =0 \}$ is invariant under \eqref{eqn:Bi-systems-fix}, and hence, $B_1=B_3=0$ on the regular interval. Since this violates the asymptotics for $B_1$ near $t=0$, we must have $\beta>\Theta_k$.
\end{proof}

\noindent From the claim, it follows that near $s=0$, the solution lies in the cone $\mathcal{C}_B := \{ B_1 \geq 0, B_2 \leq 0, B_3 \geq 0 \}$. Since this cone is backward invariant, the solution lies in this cone for the whole regular interval. The inequalities are strict on compact subintervals of $(0,T)$. To show that $b_1=b_2=b_3$, set $E_1 : = b_2-b_1$ and $E_3 : = b_2 - b_3$. From \eqref{eqn:bi-ode}, we have  \begin{eqnarray*}
    E_1' &=& (B_2-B_1)E_3-(B_2+2B_3)E_1, \\
    E_3' &=& (B_2-B_3)E_1-(2B_1+B_2)E_3.
\end{eqnarray*} Near $s =0$, \begin{eqnarray*}
    E_1, E_3 \ = \ \frac{3\Theta_k-\beta}{2q^2} + O(s). 
\end{eqnarray*} Suppose $\beta = 3\Theta_k$. Set $\widetilde{E}_i : = E_i(T-s)$ and write $\textbf{X}_{\widetilde{E}}(s) : = \begin{pmatrix}
    \widetilde{E}_1(s) \\ \widetilde{E}_3(s)
\end{pmatrix}$. The asymptotics near $s=0$ allow us to write \begin{eqnarray*}
    s \textbf{X}_{\widetilde{E}}'(s) &=& \begin{pmatrix}
-\frac{k+2}{2} & \frac{k+2}{2}\\
\frac{k+2}{2} & -\frac{k+2}{2}
\end{pmatrix} \textbf{X}_{\widetilde{E}}(s) + s \textbf{B}_{\widetilde{E}}(s) \textbf{X}_{\widetilde{E}}(s).
\end{eqnarray*} The eigenvalues of the indicial matrix are $-(k+2)$ and $0$. Since $\textbf{X}_{\widetilde{E}}=O(s)$, Lemma~\ref{lem:regular-singular-germ} implies $E_1=E_3=0$ for all $s>0$ sufficiently small. By invariance, this holds on the whole regular interval.  Suppose now that $3\Theta_k - \beta \neq 0$ and define $\varepsilon : = \text{sign}(3\Theta_k-\beta)$. Then for $s>0$ sufficiently small, the solution lies in $\mathcal{C}_E : = \{ \varepsilon E_1 \geq 0, \varepsilon E_3 \geq 0 \}$. The cone $\mathcal{C}_E$ is backward invariant under the evolution equations for $E_1,E_3$, and hence, the solution lies in $\mathcal{C}_E$ on the whole regular interval. Near $t=0$, $E_3 = b_2 - b_3 \to 0$, while 
\[
2B_1+B_2=\frac1t+O(1).
\]
The integrating factor argument used in Theorem~\ref{thm:SO(3)-S4} can be followed almost verbatim to show that $E_1=E_3=0$ on the whole regular interval. Hence, in all cases, $b_1=b_2=b_3 = \frac{\lambda}{3}$, and hence, $\mathfrak{R} \vert_{\Lambda_-^2} = \frac{\lambda}{3} \text{Id}_{\Lambda_-^2}$. The metric g is therefore self-dual. From \cite[Theorem 12.1]{GroveWilkingZiller2008}, the metric is homothetic to Hitchin's orbifold metric \cite{Hitchin1996NewFamilyEinsteinMetrics}. 
\end{proof}

\small 

\subsection*{Funding.} The author was supported by the Australian Research Council Discovery Project DP220102530.

 \normalsize

\bibliographystyle{alpha-author}
\bibliography{ref-clean}

\end{document}